\documentclass[12pt,a4paper,reqno]{amsart}

\usepackage[all,2cell,arrow,color]{xy}
\newdir{ >}{{}*!/-10pt/@{>}}
\usepackage[english]{babel}
\usepackage{mathtools}
\usepackage{stmaryrd}
\usepackage{tensor}
\numberwithin{equation}{section}
\usepackage{graphicx} 
\usepackage{rotating}
\usepackage{pdflscape}
\usepackage[pagebackref,colorlinks,urlcolor=darkgreen,citecolor=darkgreen,linkcolor=darkgreen]{hyperref}
\usepackage{enumerate,xspace}
\usepackage{xcolor}

\definecolor{darkgreen}{rgb}{0,0.45,0}

  \newtheorem{proposition}{Proposition}[section]
  \newtheorem{lemma}[proposition]{Lemma}
  \newtheorem{corollary}[proposition]{Corollary}

  \theoremstyle{definition}
  \newtheorem{definition}[proposition]{Definition}
  \newtheorem{example}[proposition]{Example}

  \newtheorem{assumption}[proposition]{Assumption}
  
\theoremstyle{remark}

  \newcounter{c}
  \renewcommand{\[}{\setcounter{c}{1}$$}
  \newcommand{\etyk}[1]{\vspace{-7.4mm}$$\begin{equation}\Label{#1}
  \addtocounter{c}{1}}
  \renewcommand{\]}{\ifnum \value{c}=1 $$\else \end{equation}\fi}
\makeatletter
\newcommand*{\inlineequation}[2][]{%
  \begingroup
    \refstepcounter{equation}%
    \ifx\\#1\\%
    \else
      \label{#1}%
    \fi
    \relpenalty=10000 %
    \binoppenalty=10000 %
    \ensuremath{%
      #2%
    }%
    ~\@eqnnum
  \endgroup
}

\makeatletter
\def\@settitle{\begin{center}%
  \baselineskip14\p@\relax
  \bfseries
  \uppercasenonmath\@title
  \@title
  \ifx\@subtitle\@empty\else
     \\[1ex]\uppercasenonmath\@subtitle
     \footnotesize\mdseries\@subtitle
  \fi
  \end{center}%
}
\def\subtitle#1{\gdef\@subtitle{#1}}
\def\@subtitle{}
\makeatother

\newcommand{\coten}[1]{\raisebox{-7pt}{\ensuremath{\stackrel{\displaystyle  \Box}{\scriptstyle { #1}}}}}
\newcommand{\morcoten}{\scalebox{.7}{\ensuremath{\,\Box \,}}}
\newcommand{\diagcoten}{\scalebox{.55}{\ensuremath \Box}}
\newcommand{\expcoten}[1]{\scalebox{.55}{
{\raisebox{-8pt}{\ensuremath{\stackrel{\displaystyle  \Box}{\scriptstyle { #1}}}}}}}

\begin{document}

\title[Crossed modules of monoids I.]{Crossed modules of monoids I.}
\subtitle{Relative categories}

\author{Gabriella B\"ohm} 
\address{Wigner Research Centre for Physics, H-1525 Budapest 114,
P.O.B.\ 49, Hungary}
\email{bohm.gabriella@wigner.mta.hu}
\date{March 2018}
%\subjclass{}
  
\begin{abstract}
This is the first part of a series of three strongly related papers in which
three equivalent structures are studied:
\begin{itemize}
\item[-] internal categories in categories of monoids; defined in terms of
pullbacks relative to a chosen class of spans
\item[-] crossed modules of monoids relative to this class of spans 
\item[-] simplicial monoids of so-called Moore length 1 relative to this class
  of spans. 
\end{itemize}
The most important examples of monoids that are covered are small categories
(treated as monoids in categories of spans) and bimonoids in symmetric
monoidal categories (regarded as monoids in categories of comonoids). 
In this first part the theory of relative pullbacks is worked out
leading to the definition of a relative category. 
\end{abstract}
  
\maketitle

%%%%%%%%%%%%%%%%%%%%%%%%%%%%%%%  INTRODUCTION  %%%%%%%%%%%%%%%%%%%%%%%%%%%%%%

\section*{Introduction} \label{sec:intro}

Loosely speaking, a {\em crossed module of a group} \cite{Whitehead} looks
like a normal subgroup but it needs not be an inclusion in general. Its
significance stems from its relation to various structures: a {\em simplicial
group whose Moore complex is concentrated in degrees 1 and 2} will be the
internal nerve of a {\em strict 2-group} and the Moore complex will be the
corresponding {\em crossed module}. These constructions establish, in fact,
equivalences between these three notions.  
Via the above links, crossed modules found diverse applications: in 
combinatorial homotopy, 
differential geometry,
the theory of classifying spaces, 
in non-abelian cohomology and even in 
(mathematical) physics, in topological and homotopical quantum field theories
\cite{Whitehead_I, Whitehead_II, BaezLauda, BreenMessing, Breen,
BrownHigginsSivera, Baues, MacLane, Quillen, Andre, Yetter:QFT, Yetter:TQFT,
DijkgraafWitten, PorterTuraev, Porter:TQFT, Porter:HQFT, Bantay}. Nice surveys
can be found in \cite{Paoli,Porter:Menagerie}.  

A proof of the equivalence between crossed modules and strict 2-groups (that is, of internal categories in the category of groups) can be found in \cite{BrownSpencer}, where it is referred also to an unpublished proof \cite{Duskin}.
Based on purely category theoretical arguments, using the semi-Abelian  structure of the category of groups, in \cite{Janelidze} George Janelidze gave another concise and highly elegant proof. An extensive analysis in  the semi-Abelian context was carried out in \cite{TVdL}.

Groups can be thought of as the Hopf monoids in the Cartesian monoidal
category of sets. 
Indeed, in any {\em monoidal category} one can discuss {\em monoids}
(i.e. objects equipped with an associative and unital
multiplication). Ordinary monoids are re-covered as monoids in the Cartesian
monoidal category of sets. Dually, one can define {\em comonoids} in arbitrary
monoidal categories as monoids in the opposite category. In Cartesian monoidal
categories every object has a unique comonoid structure so this gives nothing
interesting in the category of sets. Whenever a monoidal category is {\em
braided} as well --- that is, there is a natural isomorphism allowing to
switch the order of the factors in the monoidal product --- both monoids and
comonoids of this monoidal category constitute monoidal categories. Using this
fact, one can define {\em bimonoids} as monoids in the category of comonoids;
equivalently, as comonoids in the category of monoids. Again, if the monoidal
structure is Cartesian (e.g. in the category of sets) this gives nothing new:
bimonoids coincide with monoids. {\em Hopf monoids} in braided monoidal
categories are distinguished bimonoids for which a canonical morphism is
invertible. Hopf monoids in the category of sets are precisely the groups. 
Hopf monoids have been studied most intensively in the category of
vector spaces where they are known as {\em Hopf algebras}. 

Motivated by various applications, some research on {\em crossed modules of
Hopf algebras} \cite{Majid,FariaMartins} and of more general {\em Hopf monoids}
\cite{Aguiar,Villanueva} began. In these papers, crossed modules of Hopf
monoids were related to category-like objects in the category of Hopf
monoids. Most recently, in \cite{Emir} crossed modules over cocommutative Hopf
algebras were related to cocommutative simplicial Hopf algebras with length 2
Moore complex (using arguments based on direct computation).  

While Janelidze's approach in \cite{Janelidze} via  semi-Abelian categories
gives a very short proof and a very clear explanation of the equivalence
between internal categories and crossed modules, it is not directly applicable
to categories of Hopf monoids in arbitrary braided monoidal categories. While
groups constitute a semi-abelian category, general Hopf monoids do not (see
however \cite{GKV}). 
In order to obtain a theory which is conceptually as clear as
\cite{Janelidze}, but has a wider application, in the current series of papers
we develop a theory dealing with monoids in general, not necessarily Cartesian
monoidal categories. 
In this way we recover two classes of examples: 
\begin{itemize}
\item In the paper \cite{BrownIcen} one can find the definition of crossed
  modules of {\em groupoids}, which is generalized to any categories in a
  straightforward way. Regarding small categories as monoids in categories of
  spans, in our theory we re-obtain the crossed modules of small categories as
  a particular case. 
\item In \cite{Villanueva} one can find the definition of crossed
  modules of Hopf monoids in symmetric monoidal categories, which is again
  smoothly generalized to bimonoids. Regarding bimonoids as monoids in
  categories of comonoids, in our theory we re-obtain the crossed modules of
  bimonoids (so in particular of ordinary monoids in the category of sets) as
  a particular case.  
  Placing the results of \cite{Villanueva} in our more general framework, we also find a conceptual reason why they only hold in a {\em symmetric} monoidal category, what obstructs the generalization to an arbitrary braiding.
\end{itemize} 

In carrying out our programme, the first question to understand is what to
mean by an internal category 
in categories where arbitrary pullbacks may not exist (note the lack of
pullbacks in categories of comonoids of our main interest).
Resolving this problem, in this first part of the series we propose some
`admissibility' axioms on a class of spans and define pullbacks relative to
such a class $\mathcal S$. Assuming that relative pullbacks of those cospans
whose `legs are in $\mathcal S$' -- a terminology to be made precise later in
Definition \ref{def:legs_in_S} -- exist, as they do in the examples in our
mind, we obtain a monoidal category whose objects are the spans with their
legs in $\mathcal S$. An $\mathcal S$-relative category is meant then to be a
monoid therein.  

Working in a monoidal category  $\mathsf C$, we may require the compatibility of
our admissible class $\mathcal S$ of spans with the monoidal structure. With
this compatibility at hand, $\mathcal S$ induces an admissible class of spans
in the category of monoids in $\mathsf C$; hence relative categories in the
category of monoids are available. 
In Part II of this series \cite{Bohm:Xmod_II} their category is shown to be
equivalent to the category of relative crossed modules of monoids in a
suitable sense; and in Part III \cite{Bohm:Xmod_III} to the category of
relative simplicial monoids of so-called Moore length 1. 

Extending the picture on crossed modules (of groups) recalled above, {\em
$n$-crossed modules}  
can be seen as Moore complexes of {\em simplicial groups}, concentrated in
degrees up-to $n+1$; and such simplicial groups arise as suitable nerves of
{\em $\mathsf{Cat}^n$-groups} (i.e. $n$-fold categories in
the category of groups). Again, these correspondences are in fact equivalences
\cite{Conduche,Porter}. 
These equivalent viewpoints are both of conceptual and
practical use: each of them gives a different insight and interpretation of
the same thing; and they provide the possibility for finding the (sometimes
technically) smoothest approach in the applications
\cite{Conduche,EllisSteiner,Loday,Guin-Walery,Porter}.  
We believe that our methods should be suitable to obtain an analogous theory
of higher relative crossed modules of monoids what we plan to discuss
elsewhere.  

\subsection*{Acknowledgement} 
The author's interest in the subject was triggered by the excellent workshop
{\em `Modelling Topological Phases of Matter -- TQFT, HQFT, premodular and
  higher categories, Yetter-Drinfeld and crossed modules in disguise'} in
Leeds UK, 5-8 July 2016. It is a pleasure to thank the organizers,
Zolt\'an K\'ad\'ar, Jo\~ao Faria Martins, Marcos Cal\c{c}ada and Paul Martin
for the experience and a generous invitation.
Financial support by the Hungarian National Research, Development and
Innovation Office – NKFIH (grant K124138) is gratefully acknowledged.  

%%%%%%%%%%%%%%%%%%%%%%%%%%%%%%% SEC 1   %%%%%%%%%%%%%%%%%%%%%%%%%%%%%%

\section{Preliminaries on monoids in monoidal categories}

In this preliminary section we recall --- without, or with very sketchy proofs
--- some known facts about monoids that will play important roles in our later
constructions; in particular Part II. Throughout the section $\mathsf M$ denotes a monoidal category whose monoidal unit is $I$ and the monoidal product is denoted by juxtaposition. For the monoidal product of $n$ copies of the same object  $A$ also the power notation $A^n$ is used. The monoidal structure is not assumed to be strict but the associativity and unit coherence isomorphisms are not explicitly denoted. Whenever $\mathsf M$ is assumed to be braided monoidal, its braiding will be denoted by $c$. Composition of morphisms $f:A\to B$ and $g:B\to C$ is denoted by $g.f:A \to C$ and identity morphisms are denoted by $1$.

\begin{definition}\label{def:monoid}
A {\em monoid}  in $\mathsf{M}$  consists of an object $A$ together with a multiplication morphism 
$\xymatrix@C=12pt{
A^2\ar[r]^-m & 
A}$ 
and a unit morphism 
$\xymatrix@C=12pt{
I \ar[r]^-u & 
A}$ 
such that the associativity condition $m.m1=m.1m$ and the unit conditions $m.u1=1=m.1u$ hold (note the omitted coherence isomorphisms).
A {\em monoid morphism} is a morphism 
$\xymatrix@C=12pt{
A \ar[r]^-f & A'}$ 
for which $f.m=m'.ff$ and $f.u=u'$.
\end{definition}

\begin{lemma} \label{lem:joint-epi}
Two monoid morphisms 
$\xymatrix@C=15pt{
A \ar[r]|(.43){\, f\,} & C & B \ar[l]|(.43){\, g\,}}$
in $\mathsf M$ are joint epimorphisms of monoids whenever the induced morphism 
\begin{equation} \label{eq:q}
q:=
\xymatrix{
AB \ar[r]^-{fg} &
C^2 \ar[r]^-m &
C}
\end{equation}
is an epimorphism in $\mathsf M$.
\end{lemma}

\begin{proof}
If $x.f=y.f$ and $x.g=y.g$ for some parallel monoid morphisms $x$ and $y$, then also $x.q=y.q$. 
\end{proof}

\begin{definition}\label{def:dlaw}
A {\em distributive law} in $\mathsf M$ consists of two monoids $A$ and $B$ together with a morphism 
$\xymatrix@C=12pt{
BA \ar[r]^-x & AB}
$ 
such that the following identities hold.
\begin{eqnarray*}
x.m1=1m.x1.1x \qquad & x.u1=1u \\
x.1m=m1.1x.x1 \qquad & x.1u=u1
\end{eqnarray*}
\end{definition}

\begin{lemma} \label{lem:product-monoid}
For any distributive law 
$\xymatrix@C=12pt{
BA \ar[r]^-x & AB,}$ 
there is an induced monoid with object part $AB$, unit 
$\xymatrix@C=15pt{
I \ar[r]^-{uu}
&  AB}$ 
and multiplication 
$\xymatrix{
(AB)^2 \ar[r]^-{1x1} &
A^2B^2  \ar[r]^-{mm} &
AB.}$
For this monoid both 
$\xymatrix@C=15pt{A \ar[r]^-{1u} & AB}$ and
$\xymatrix@C=15pt{B \ar[r]^-{u1} & AB}$ are monoid morphisms.

\end{lemma}

\begin{lemma} \label{lem:monoid-factorization}
Consider monoid morphisms 
$\xymatrix@C=15pt{
A \ar[r]|(.43){\, f\,} & C & B \ar[l]|(.43){\, g\,}
}$
such that the induced morphism 
\eqref{eq:q} is invertible. Then the unique monoid structure on $AB$ for which $q$ is a monoid morphism is induced by the distributive law 
$$
\xymatrix{
BA \ar[r]^-{gf} &
C^2 \ar[r]^-m &
C \ar[r]^-{q^{-1}} &
AB.}
$$
\end{lemma}

\begin{lemma} \label{lem:monoid-morphism}
For a distributive law 
$\xymatrix@C=12pt{
BA \ar[r]^-x & AB
}$ 
and a monoid $C$, there is a bijective correspondence between the following data.
\begin{itemize}
\item[{(i)}] monoid morphisms 
$\xymatrix@C=12pt{
AB \ar[r]^-c &
C
}$ (where the monoid structure of $AB$ is induced by $x$)
\item[{(ii)}] pairs of monoid morphisms 
$(
\xymatrix@C=12pt{
A \ar[r]^-a & C},
\xymatrix@C=12pt{
B \ar[r]^-b & C}
)$
such that $m.ab.x=m.ba$.
\end{itemize}
\end{lemma}

\begin{proof}
A monoid morphism $c$ in part (i) is sent to the pair $(c.1u,c.u1)$. Conversely, a pair $(a,b)$ in part (ii) is sent to $m.ab$.
\end{proof}

\begin{corollary} \label{cor:factorising-morphism}
Consider monoid morphisms 
$\xymatrix@C=15pt{
A \ar[r]|(.43){\, f\,} & C & B \ar[l]|(.43){\, g\,}
}$
such that the induced morphism 
\eqref{eq:q} is invertible. For any monoid $D$, there is a bijective correspondence between the following data.
\begin{itemize}
\item[{(i)}] monoid morphisms 
$\xymatrix@C=12pt{
C \ar[r]^-c &
D}$
\item[{(ii)}] pairs of monoid morphisms 
$(
\xymatrix@C=12pt{
A \ar[r]^-a & D},
\xymatrix@C=12pt{
B \ar[r]^-b & D}
)$
such that $m.ab.q^{-1}.m.gf=m.ba$.
\end{itemize}
The morphism $c$ in part (i) is equal to 
$\xymatrix@C=15pt{C\ar[r]^-{q^{-1}} & AB \ar[r]^-{ab} & D^2 \ar[r]^-m & D}$.
It is the unique simultaneous solution of the equations $c.f=a$ and $c.g=b$.
\end{corollary}

%%%%%%%%%%%%%%%%%%%%%%%%%%%%%%% SEC 2   %%%%%%%%%%%%%%%%%%%%%%%%%%%%%%

\section{Admissible classes of spans}
\label{sec:S}

We are interested in categories --- like the categories of comonoids in
symmetric monoidal categories, see the Introduction --- in which general
pullbacks may not exist. Instead, we will assume the existence of certain {\em
relative} pullbacks with respect to some distinguished class of spans. By
this motivation in this section we investigate the expected properties of such
a class.

\begin{definition} \label{def:admissibleS}
A class $\mathcal S$ of spans in any category $\mathsf C$ is said to be {\em admissible} if it satisfies the following two conditions. 
\begin{itemize} 
\item[{(PO}]\hspace{-.18cm}ST)
If  
$\xymatrix@C=12pt{
X & \ar[l]_-f A \ar[r]^-g & Y}
\in {\mathcal S}$ 
then 
$\xymatrix@C=12pt{
X' &  
\ar[l]_-{f'} X & 
\ar[l]_-f A \ar[r]^-g & 
Y \ar[r]^-{g'} &
Y'}
\in {\mathcal S}$  
too, for any morphisms 
$\xymatrix@C=12pt{
X \ar[r]^-{f'} & X'} 
\textrm{ and }
\xymatrix@C=12pt{
Y \ar[r]^-{g'} & Y'}$.
\item[{(PR}]\hspace{-.17cm}E)
If  
$\xymatrix@C=12pt{
X & \ar[l]_-f A \ar[r]^-g & Y}
\in {\mathcal S}$ 
then 
$\xymatrix@C=12pt{
X & 
\ar[l]_-f A &
\ar[l]_-h B \ar[r]^-h &
A \ar[r]^-g & Y}
\in {\mathcal S}$ 
too, for any morphism 
$\xymatrix@C=12pt{
B \ar[r]^-h & A}$.
\end{itemize}
\end{definition}

\begin{example} \label{ex:allS}
The class of all spans in a category is clearly admissible.
\end{example}

\begin{example} \label{ex:CoMonS}
For a monoidal category $\mathsf M$, let $\mathsf C$ be the
category of comonoids in $\mathsf M$ (that is, the category of monoids in the
monoidal category $\mathsf{M}^{\mathsf{rev}}$ with the opposite composition). 
Assume that $\mathsf M$ is braided monoidal (with braiding $c$). Then $\mathsf
C$ inherits the monoidal structure of $\mathsf M$: the monoidal unit $I$ is a
trivial comonoid with comultiplication $I\cong I^2$ provided by the unit
isomorphisms, and the monoidal product $AC$ of any comonoids $A$ and $C$ is a
comonoid via the comultiplication
$\xymatrix@C=12pt{AC \ar[r]^-{\delta \delta} & A^2C^2\ar[r]^-{1c1}& (AC)^2}$
induced by the comultiplications
$\xymatrix@C=12pt{A \ar[r]^-\delta & A^2}$ and
$\xymatrix@C=12pt{C \ar[r]^-\delta & C^2}$.

Define $\mathcal S$ to contain precisely those spans
$\xymatrix@C=12pt{
X & \ar[l]_-f A \ar[r]^-g & Y}$
in $\mathsf C$ for which the composite morphism
$\xymatrix@C=12pt{A \ar[r]^-\delta & A^2 \ar[r]^-{fg} & XY}$ 
is a {\em comonoid morphism}; equivalently, the equality
$c.fg.\delta=gf.\delta$ holds.

For any comonoid morphisms $f'$ and $g'$ of respective domains $X$ and $Y$,
the monoidal product $f'g'$ is a comonoid morphism. Then so is
$f'g'.fg.\delta$ for any 
$\xymatrix@C=12pt{
X & \ar[l]_-f A \ar[r]^-g & Y}
\in {\mathcal S}$, so that condition (POST) is satisfied.

On the other hand, for any comonoid morphism $h$ of codomain $A$,
$fg.\delta.h=fg.hh.\delta$ is a comonoid morphism so also (PRE) holds. 

The current example can be considered in the particular situation when
$\mathsf M$ is a Cartesian monoidal (so symmetric monoidal) category. Then
every object has a unique comonoid structure; that is, $\mathsf C$ and
$\mathsf M$ are isomorphic. In particular, every comonoid is cocommutative
(that is, the comultiplication $\delta$ and the symmetry $c$ satisfy
$c.\delta=\delta$). Then the class $\mathcal S$ of spans above is the class of
all spans in $\mathsf C \cong \mathsf M$.  
\end{example}

\begin{lemma} \label{lem:splitS}
Let $\mathcal S$ be an admissible class of spans in an arbitrary category $\mathsf C$ and let 
$\xymatrix@C=15pt{
A \ar@<2pt>@{ >->}[r]^-i &
B\ar@<2pt>@{->>}[l]^-s}$
be a split epimorphism in $\mathsf C$.
\begin{itemize}
\item[{(1)}] The following assertions are equivalent.
\begin{itemize}
\item[{(a)}] 
$\xymatrix@C=12pt{
B & \ar@{=}[l] B \ar@{=}[r]  &B}
\in \mathcal S$.
\item[{(b)}] 
$\xymatrix@C=12pt{
X & \ar[l]_-f B \ar[r]^-g & Y}
\in \mathcal S$ for any morphisms $f$ and $g$ of domain $B$.
\item[{(c)}]  
$\xymatrix@C=12pt{
A & \ar[l]_-i B  \ar@{=}[r]  &B}
\in \mathcal S$.
\end{itemize}
\item[{(2)}] The equivalent assertions of part (1) hold whenever 
$\xymatrix@C=12pt{
A & \ar@{=}[l] A \ar[r]^-s & B}
\in \mathcal S$.
\end{itemize}
\end{lemma}

\begin{proof}
Assertion (a) of part (1) implies (b) by condition (POST) on $\mathcal S$, post-composing by $f$ on the left and by $g$ on the right. 
Assertion (b) trivially implies (c). Finally, (c) implies (a) by (POST), post-composing on the left by $s$ and using $s.i=1$. 
The condition in part (2) implies (c) of part (1)  by (PRE), pre-composing by $i$ and using $s.i=1$ again.
\end{proof}

\begin{definition} \label{def:monoidalS}
A class ${\mathcal S}$ of spans in a monoidal category $\mathsf M$ is said to be {\em monoidal} if it satisfies the following two conditions.
\begin{itemize}
\item[{(UNI}] \hspace{-.35cm} TAL)
For any morphisms $f$ and $g$ whose domain is the monoidal unit $I$,
$\xymatrix@C=12pt{
X & \ar[l]_-f I \ar[r]^-g & Y}
\in {\mathcal S}$.
\item[{(MU}] \hspace{-.3cm} LTIPLICATIVE)
If both
$\xymatrix@C=12pt{
X & \ar[l]_-f A \ar[r]^-g & Y}
\in {\mathcal S}$ 
and
$\xymatrix@C=12pt{
X' & \ar[l]_-{f'} A' \ar[r]^-{g'} & Y'}
\in {\mathcal S}$ 
then also
$\xymatrix@C=12pt{
XX' & \ar[l]_-{ff'} A A'\ar[r]^-{gg'} & YY'}
\in {\mathcal S}$.
\end{itemize}
\end{definition}

A class of spans satisfying (POST) is unital if and only if 
$\xymatrix@C=12pt{
I & \ar@{=}[l] I \ar@{=}[r] & I}
\in {\mathcal S}$.

\begin{example} \label{ex:allSmonoidal}
The class of all spans in a monoidal category is clearly monoidal.
\end{example}

\begin{example} \label{ex:CoMonSmonoidal}
For a braided monoidal category $\mathsf M$ (with braiding $c$) let $\mathsf C$ be the category of comonoids in $\mathsf M$.  It is monoidal via the monoidal product of $\mathsf M$, see Example \ref{ex:CoMonS}. 
Below we show that  the class $\mathcal S$ in Example \ref{ex:CoMonS} of spans in $\mathsf C$ is monoidal whenever the symmetry is a braiding; that is, $c^{-1}=c$. (This explains in a conceptual way why in \cite{Villanueva} it is dealt only with symmetric monoidal categories not with arbitrary braidings.)

By the coherence of the braiding, the comultiplication $\delta$ of $I$ satisfies $c.\delta=\delta$. Then 
$\xymatrix@C=12pt{
I & \ar@{=}[l] I \ar@{=}[r] & I}
\in {\mathcal S}$, and so the unitality of $\mathcal S$ follows by its property (POST), see Example \ref{ex:CoMonS}.

For any 
$\xymatrix@C=12pt{
X & \ar[l]_-f A \ar[r]^-g & Y}
\in {\mathcal S}$ 
and
$\xymatrix@C=12pt{
X' & \ar[l]_-{f'} A' \ar[r]^-{g'} & Y'}
\in {\mathcal S}$ 
the following diagram commutes.
$$
\xymatrix{
(AA')^2 \ar[dd]_-{ff'gg'} &
A^2A^{\prime 2} \ar[dd]^-{fgf'g'} \ar[l]_-{1c1} &
AA' \ar[l]_-{\delta\delta'} \ar[r]^-{\delta\delta'} &
A^2A^{\prime 2} \ar[d]^-{gfg'f'} \ar[r]^-{1c1} &
(AA')^2 \ar[d]^-{gg'ff'} \\
&&& YXY'X' \ar[r]^-{1c1} \ar[d]^-{c^{-1}c^{-1}} &
YY'XX' \ar[d]^-{c^{-1}_{XX',YY'}} \\
XX'YY' &
XYX'Y' \ar[l]^-{1c1} \ar@{=}[rr] &&
XYX'Y' \ar[r]_-{1c^{-1}1} &
XX'YY'}
$$
If $c$ is a symmetry, then the arrows of the bottom row are equal isomorphisms proving the equality of the top-left and the top-right paths; that is, 
$\xymatrix@C=12pt{
XX' & \ar[l]_-{ff'} AA' \ar[r]^-{gg'} & YY'}
\in {\mathcal S}$.
\end{example}

\begin{example} \label{ex:Mon(M)S}
Consider any class $\mathcal S'$ of spans in an arbitrary category $\mathsf C'$. For any functor $U:\mathsf C \to \mathsf C'$ define the class $\mathcal S$ which contains precisely those spans in $\mathsf C$ whose image belongs to $\mathcal S'$. 
\begin{itemize}
\item[{(1)}] If $\mathcal S'$ is admissible then so is $\mathcal S$.
\item[{(2)}]Assume that $\mathsf C$ and  $\mathsf C'$ are monoidal categories and $U$ is a strict monoidal functor. If $\mathcal S'$ is monoidal then so is $\mathcal S$.
\end{itemize}

For (1) note that 
$\xymatrix@C=12pt{
X & \ar[l]_-f A \ar[r]^-g & Y}
\in {\mathcal S}$ 
if and only if
$\xymatrix@C=12pt{
UX & \ar[l]_-{Uf} UA \ar[r]^-{Ug} & UY}
\in {\mathcal S'}$.
If this is the case, then by property (POST) of  ${\mathcal S'}$, also
$\xymatrix@C=12pt{
UX' &
UX \ar[l]_-{Uf'} & 
\ar[l]_-{Uf} UA \ar[r]^-{Ug} & 
UY \ar[r]^-{Ug'} &
UY'}
\in {\mathcal S'}$
for all morphisms $f'$ and $g'$ in $\mathsf C$ with respective domains $X$ and $Y$. By definition this is equivalent to
$\xymatrix@C=12pt{
X' &
X \ar[l]_-{f'} & 
\ar[l]_-{f} A \ar[r]^-{g} & 
Y \ar[r]^-{g'} &
Y'}
\in {\mathcal S}$
proving property (POST) of ${\mathcal S}$.  Analogous reasoning applies to property (PRE).

For (2) observe that for any span 
$\xymatrix@C=12pt{
X & \ar[l]_-f I  \ar[r]^-g & Y}$
in $\mathcal C$,
$\xymatrix@C=12pt{
UX & \ar[l]_-{Uf}UI  =I' \ar[r]^-{Ug} & UY}
\in {\mathcal S'}$
by the unitality of ${\mathcal S'}$. Then 
$\xymatrix@C=12pt{
X & \ar[l]_-f I  \ar[r]^-g & Y}
\in \mathcal S$  by definition so that $\mathcal S$ is unital.

For 
$\xymatrix@C=12pt{
X & \ar[l]_-f A \ar[r]^-g & Y}
\in {\mathcal S}$ 
and
$\xymatrix@C=12pt{
X' & \ar[l]_-{f'} A' \ar[r]^-{g'} & Y'}
\in {\mathcal S}$,
$\xymatrix@C=12pt{
UX & \ar[l]_-{Uf} UA \ar[r]^-{Ug} & UY}
\in {\mathcal S'}$ 
and
$\xymatrix@C=12pt{
UX' & \ar[l]_-{Uf'} UA' \ar[r]^-{Ug'} & UY'}
\in {\mathcal S'}$ by definition. Then by the multiplicativity of $\mathcal S'$, also 
$$
\xymatrix@C=20pt{
U(XX')=(UX )(UX') & 
\ar[l]_-{\raisebox{10pt}{${}_{U(ff')=(Uf)(Uf')}$}} U(AA')=U(A)U(A')
 \ar[r]^-{\raisebox{10pt}{${}_{U(gg')=(Ug)(Ug')} $}} & 
U(YY')=(UY)(UY')}
\in {\mathcal S'}.
$$
By definition this is equivalent to  
$\xymatrix@C=15pt{
XX' & \ar[l]_-{ff'}A A' \ar[r]^-{gg'} & YY'}
\in {\mathcal S}$, proving the multiplicativity of $\mathcal S$.

In particular, consider a monoidal category $\mathsf M$ and a class $\mathcal S'$ of spans in $\mathsf M$. Take $\mathsf C$ to be the category of monoids in $\mathsf M$ and  $\mathcal S$ to be the class containing precisely those spans in $\mathsf C$ whose image under the forgetful functor $U:\mathsf C \to \mathsf M$ belongs to $\mathcal S'$. 
\begin{itemize}
\item[{(1)}] If $\mathcal S'$ is admissible then so is $\mathcal S$.
\item[{(2)}] Assume that $\mathsf M$ is a braided monoidal category (so that also $\mathsf C$ is monoidal and $U$ is strict monoidal). If  $\mathcal S'$ is monoidal then so is $\mathcal S$.
\end{itemize}
\end{example}

\begin{definition}\label{def:legs_in_S}
For any class $\mathcal S$ of spans in some category $\mathsf C$ we say that a
cospan 
$\xymatrix@C=12pt{A \ar[r]^-f & B & \ar[l]_-g C}$
{\em has legs in $\mathcal S$} if 
$\xymatrix@C=12pt{A & \ar@{=}[l] A \ar[r]^-f & B}$ and
$\xymatrix@C=12pt{B & \ar[l]_-g C \ar@{=}[r] & C}$ belong to $\mathcal S$.

A span $\xymatrix@C=12pt{B &\ar[l]_-g A \ar[r]^-f & B}$ (with equal objects at
the left and the right) is said to {\em have its legs in $\mathcal S$} if the
cospan  
$\xymatrix@C=12pt{A \ar[r]^-f & B & \ar[l]_-g A}$
has legs in $\mathcal S$.
\end{definition}

%%%%%%%%%%%%%%%%%%%%%%%%%%%%%%% SEC 3   %%%%%%%%%%%%%%%%%%%%%%%%%%%%%%

\section{Relative pullbacks}
\label{sec:rel_pullback}

\begin{definition} \label{def:S-pullback}
Consider an admissible class $\mathcal S$ of spans in an arbitrary category $\mathsf C$. For some morphisms 
$\xymatrix@C=12pt{
A  \ar[r]^-a & B & \ar[l]_-c C}$
in $\mathsf C$ the {\em relative pullback} with respect to $\mathcal S$ --- if it exists --- is a span 
$\xymatrix@C=15pt{
A & \ar[l]_-{p_A} A \coten B C \ar[r]^-{p_C} & C}$
belonging to $\mathcal S$ such that the following properties hold (see diagram
\eqref{eq:S-pullback} below).
\begin{itemize}
\item[{(i)}] $a.p_A=c.p_C$
\item[{(ii)}] {\em Universality}: for any 
$\xymatrix@C=15pt{
A & \ar[l]_-{f} X \ar[r]^-{g} & C}
\in \mathcal S$ such that $a.f=c.g$, there is a unique morphism 
$\xymatrix@C=12pt{
X \ar[r]^-h & A \coten B C}$  in $\mathsf C$ which satisfies $p_A.h=f$ and $p_C.h=g$.
\item[{(iii)}] {\em Reflection:} if both 
$\xymatrix@C=12pt{
A  &
\ar[l]_-{p_A} A \coten B C & 
\ar[l]_-f D \ar[r]^-g & 
E}$
and
$\xymatrix@C=12pt{
C  &
\ar[l]_-{p_C} A \coten B C & 
\ar[l]_-f D \ar[r]^-g &
E}$ belong to $\mathcal S$, then also 
$\xymatrix@C=12pt{
A \coten B C & 
\ar[l]_-f D \ar[r]^-g &
E}$ belongs to $\mathcal S$; and symmetrically, if both 
$\xymatrix@C=12pt{
E &
\ar[l]_-g D \ar[r]^-f &
A \coten B C \ar[r]^-{p_A} &
A}$
and
$\xymatrix@C=12pt{
E &
\ar[l]_-g D \ar[r]^-f &
A \coten B C \ar[r]^-{p_C} &
C}$ belong to $\mathcal S$, then also 
$\xymatrix@C=12pt{
E &
\ar[l]_-g D \ar[r]^-f &
A \coten B C}$ belongs to $\mathcal S$.
\end{itemize}
\begin{equation}\label{eq:S-pullback}
\xymatrix@C=10pt@R=10pt{
X \ar@{-->}[rd]^-{!h} \ar@/^1.1pc/[rrrd]^-g \ar@/_1.1pc/[rddd]_-f \\
& A \coten B C \ar[rr]^-{p_C} \ar[dd]_-{p_A} &&
C \ar@{..>}[dd]^-{c} \\
\\
& A \ar@{..>}[rr]_-{a} &&
B}
\end{equation}
\end{definition}

By property (PRE) of $\mathcal S$, part (ii) of Definition
\ref{def:S-pullback} implies that
$\xymatrix@C=15pt{
A & \ar[l]_-{p_A} A \coten B C \ar[r]^-{p_C} & C}$
are joint monomorphisms. Therefore the $\mathcal S$-relative pullback is unique
up-to isomorphism whenever it exists.

\begin{example} \label{ex:pullback}
If $\mathcal S$ is the class of all pullbacks in some category $\mathsf C$, then $\mathcal S$-relative pullbacks are just usual pullbacks. 
\end{example}

\begin{example} \label{ex:CoMon-pullback}
As in Example \ref{ex:CoMonS}, let $\mathsf C$ be the category of comonoids in
a monoidal category $\mathsf M$. 
Assume that $\mathsf M$ has equalizers which are preserved by taking the
monoidal product with any object. Then in $\mathsf C$ any parallel morphisms
$\xymatrix@C=12pt{A\ar@<2pt>[r]^-f \ar@<-2pt>[r]_-g & B}$
have an equalizer; computed as the equalizer
$$
\xymatrix{
E \ar[r]^-j &
A \ar@<2pt>[rr]^-{\widehat f:=1f1.\delta 1.\delta}
\ar@<-2pt>[rr]_-{\widehat g:=1g1.\delta 1.\delta} &&
ABA}
$$
in $\mathsf M$ (where $\delta$ stands for the comultiplication of $A$); see
\cite{AndDev}. Clearly, any comonoid morphism of codomain $A$ equalizes $f$
and $g$ if and only if it equalizes $\widehat f$ and $\widehat g$. So in order
to prove that $E$ is the equalizer of $f$ and $g$ in $\mathsf C$, we need to
equip $E$ with a comonoid structure so that $e$ becomes a comonoid
(mono)morphism. 

The counit is
$\xymatrix@C=12pt{E \ar[r]^-j & A\ar[r]^-\varepsilon & I}$ (where
$\varepsilon$ stands for the counit of $A$). The comultiplication is
constructed in two steps. First the universality of the equalizer in $\mathsf M$ in the
bottom row of the first serially commutative diagram below is used to
construct an auxiliary morphism $\delta_r$; and then the comultiplication
$\delta$ is constructed using the universality of the equalizer in $\mathsf M$ in the bottom
row of the second serially commutative diagram in
$$
\xymatrix{
E \ar[r]^-j \ar@{-->}[d]_-{\delta_r} &
A \ar@<2pt>[r]^-{\widehat f}\ar@<-2pt>[r]_-{\widehat g} \ar[d]^-\delta &
ABA \ar[d]^-{11\delta} \\
EA \ar[r]^-{j1}  &
A^2 \ar@<2pt>[r]^-{\widehat f1}\ar@<-2pt>[r]_-{\widehat g1}  &
ABA^2 }\qquad \qquad
\xymatrix{
E \ar[r]^-j \ar@{-->}[d]_-{\delta} \ar[rd]^-{\delta_r} &
A \ar@<2pt>[r]^-{\widehat f}\ar@<-2pt>[r]_-{\widehat g} &
ABA \ar[d]^-{\delta 11} \\
E^2 \ar[r]^-{1j}  &
EA \ar@<2pt>[r]^-{j \widehat f}\ar@<-2pt>[r]_-{j \widehat g}  &
A^2BA .}
$$

Assume furthermore that $\mathsf M$ is a braided monoidal category so that
$\mathsf C$ inherits the monoidal structure of $\mathsf M$ (cf. Example
\ref{ex:CoMonS}). Any comonoid
morphisms   
$\xymatrix@C=12pt{
A  \ar[r]^-f & B & \ar[l]_-g C}$
induce comonoid morphisms 
$\xymatrix@C=12pt{
AC \ar@<2pt>[r]^-{f\varepsilon} \ar@<-2pt>[r]_-{\varepsilon g} & B}$
(where $\varepsilon$ stands for both counits of $A$ and $C$). So we can take
their equalizer 
\begin{equation}\label{eq:comonoid_eq}
\xymatrix@C=15pt{
A \coten B C \ar[r]^-j &
AC \ar@<2pt>[r]^-{f\varepsilon}
\ar@<-2pt>[r]_-{\varepsilon g} &
B}
\end{equation}
in $\mathsf C$. Below we claim that it gives in fact the pullback
\begin{equation}\label{eq:coten-pullback}
\xymatrix@C=12pt@R=12pt{
A \coten B C \ar[r]^-j \ar[d]_-j & AC \ar[r]^-{\varepsilon 1} & C \ar[dd]^-g \\
AC \ar[d]_-{1\varepsilon} \\
A \ar[rr]_-f &&
B}
\end{equation}
relative to the admissible class $\mathcal S$ in Example
\ref{ex:CoMonS} of spans in $\mathsf C$.

The square of \eqref{eq:coten-pullback} commutes since \eqref{eq:comonoid_eq}
is a fork.
The span 
$\xymatrix{A & \ar[l]_-{1\varepsilon.j} A\coten B C \ar[r]^-{\varepsilon 1.j} & C}$
belongs to $\mathcal S$ since
$1\varepsilon \varepsilon 1.jj.\delta=j$ is a comonoid morphism by construction.
In order to check the universality of \eqref{eq:coten-pullback}, take a span 
$\xymatrix@C=12pt{
A & \ar[l]_-k D \ar[r]^-l & C }$
in $\mathcal S$ for which $f.k=g.l$. 
Then $f\varepsilon.kl.\delta=f.k=g.l=\varepsilon g.kl.\delta$. Thus
since $kl.\delta$ is a comonoid morphism by assumption,
a filler $h$ of the first diagram in
\begin{equation} \label{eq:pullback_vs_equalizer}
\xymatrix@C=12pt@R=12pt{
D \ar@/^1.1pc/[rrrd]^-l \ar@/_1.1pc/[rddd]_-k \ar@{-->}[rd]^-h \\
& A \coten B C \ar[r]_-j \ar[d]^-j &
AC \ar[r]_-{\varepsilon 1} &
C \ar[dd]^-g \\
& AC \ar[d]^-{1\varepsilon} \\
& A \ar[rr]_-f &&
B}\qquad \qquad 
\raisebox{-30pt}{$
\xymatrix@C=40pt @R=40pt{
D \ar[r]^-\delta \ar@{-->}[d]_-h &
D^2 \ar[d]^-{kl} \\
A \coten B C \ar[r]_-j &
AC \ar@<2pt>[r]^-{f\varepsilon} \ar@<-2pt>[r]_-{\varepsilon g} &
B}$}
\end{equation}
is constructed via the universality of the equalizer in $\mathsf C$ in the
bottom row of the second diagram. 
It is a comonoid morphism by construction. 
The uniqueness of a comonoid morphism $h$
rendering commutative the first diagram of \eqref{eq:pullback_vs_equalizer}
follows by the observation that any comonoid morphism $h$ making the first
diagram commute, renders commutative also the second diagram of
\eqref{eq:pullback_vs_equalizer} by the commutativity of  
$$
\xymatrix@R=15pt{
D \ar[rr]^-\delta \ar[d]_-h &&
D^2 \ar[d]_-{hh} \ar@/^1.1pc/[rdd]^-{kl} \\
A \coten B C \ar[rr]^-\delta \ar[d]_-j &&
(A \coten B C)^2 \ar[d]_-{jj} \\
AC \ar[r]^-{\delta\delta} \ar@/_1.1pc/@{=}[rrr] &
A^2C^2 \ar[r]^-{1c1} & 
(AC)^2 \ar[r]^-{1\varepsilon \varepsilon 1} &
AC.}
$$
For the reflection property on the left, assume that 
\begin{equation} \label{eq:CoMon-refl}
\xymatrix@C=12pt{A & \ar[l]_-{1\varepsilon} AC & \ar[l]_-j A\coten B C & \ar[l]_-k D \ar[r]^-l & E} \in \mathcal S
\qquad \textrm{and} \qquad
\xymatrix@C=12pt{C & \ar[l]_-{\varepsilon 1} AC & \ar[l]_-j A\coten B C & \ar[l]_-k D \ar[r]^-l & E}\in \mathcal S.
\end{equation}
Then the diagram of Figure \ref{fig:CoMon-refl} commutes (the region marked by (1) commutes by the first condition, and the region marked by (2) commutes by the second condition of \eqref{eq:CoMon-refl}). Since the right column and the bottom row of the diagram of Figure \ref{fig:CoMon-refl} are equal monomorphisms, this proves the equality of the left column and the top row; that is,
$\xymatrix@C=12pt{A\coten B C & \ar[l]_-k D \ar[r]^-l & E} \in \mathcal S$.
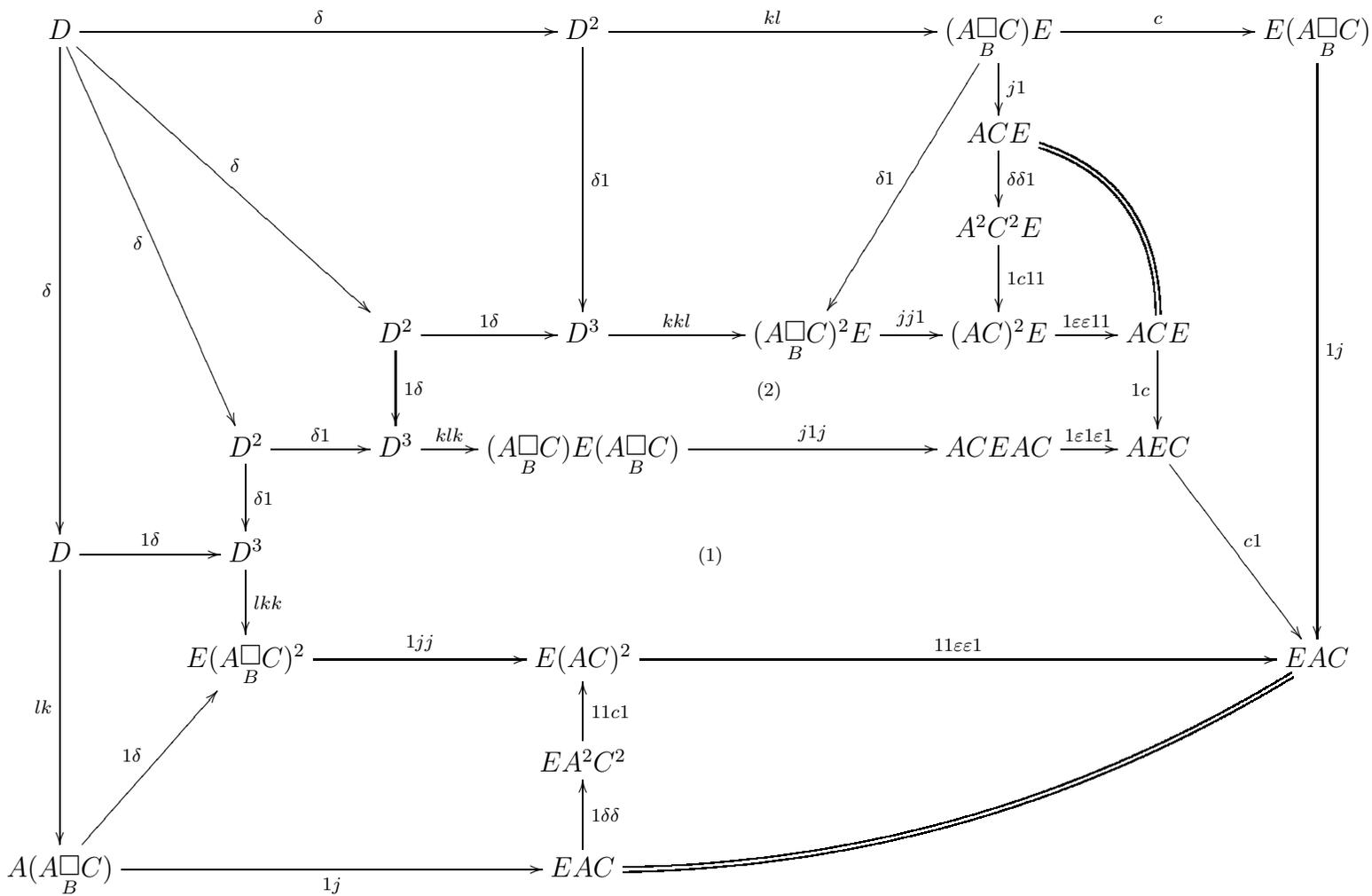
\begin{figure} 
\centering
\begin{sideways}
$\xymatrix{
% 1.1
D \ar[rrr]^-\delta \ar[rrddd]^-\delta \ar[rdddd]_-\delta \ar[ddddd]_-\delta &&&
% 1.4
D^2 \ar[rr]^-{kl} \ar[ddd]^-{\delta 1} &&
% 1.6
(A\coten B C)E\ar[lddd]_-{\delta 1} \ar[d]^-{j1} \ar[rr]^-c &&
% 1.8
E(A\coten B C) \ar[dddddd]^-{1j} \\
&&&&&
% 2.6
ACE \ar[d]^-{\delta \delta 1} \ar@{=}@/^2pc/[rdd] \\
&&&&&
% 3.6
A^2C^2E\ar[d]^-{1c11} \\
&&
% 4.3
D^2 \ar[r]^-{1\delta} \ar[d]^-{1 \delta} \ar@{}[rrrrd]|-{(2)}&
% 4.4
D^3 \ar[r]^-{kkl} &
% 4.5
(A\coten B C)^2E\ar[r]^-{jj1} &
% 4.6
(AC)^2E \ar[r]^-{1\varepsilon\varepsilon 11} &
% 4.7
ACE \ar[d]_-{1c} \\
&
% 5.2
D^2 \ar[r]^-{\delta 1} \ar[d]^-{\delta 1} \ar@{}[rrrrrdd]|-{(1)}&
% 5.3
D^3 \ar[r]^-{klk} &
% 5.4
(A\coten B C) E (A\coten B C) \ar[rr]^-{j1j} &&
% 5.6
ACEAC \ar[r]^-{1\varepsilon 1 \varepsilon 1} &
% 5.7
AEC \ar[rdd]^-{c1} \\
% 6.1
D \ar[r]^-{1\delta} \ar[ddd]_-{lk} &
% 6.2
D^3 \ar[d]^-{lkk} \\
&
% 7.2
E (A\coten B C)^2 \ar[rr]^-{1jj} &&
% 7.4
E(AC)^2 \ar[rrrr]^-{11\varepsilon \varepsilon 1} &&&&
% 7.8
EAC \\
&&&
% 8.4
EA^2C^2 \ar[u]_-{11c1} \\
% 9.1
A(A\coten B C)\ar[ruu]^-{1\delta} \ar[rrr]_-{1j} &&&
% 9.4
EAC \ar[u]_-{1\delta \delta} \ar@{=}@/_2pc/[rrrruu]}$
\end{sideways}
\caption{Reflection property}
\label{fig:CoMon-refl}
\end{figure}
A symmetrical reasoning verifies the reflection property on the right.

Summarizing, we proved that the pullback relative to the class $\mathcal
S$ in Example \ref{ex:CoMonS} of spans in $\mathsf C$ exists for any comonoid morphisms 
$\xymatrix@C=12pt{A \ar[r]^-f & B & \ar[l]_-g C}$.
It is computed as the equalizer in $\mathsf C$ of the comonoid morphisms 
$\xymatrix@C=15pt{
AC \ar@<2pt>[r]^-{f\varepsilon}\ar@<-2pt>[r]_-{\varepsilon g} & B}$.

Recall that in \cite[Definition 5]{Villanueva} for the comonoid morphisms 
$\xymatrix@C=12pt{A \ar[r]^-f & B & \ar[l]_-g C}$
it is assumed that 
$\xymatrix@C=12pt{A & \ar@{=}[l] A \ar[r]^-f & B}$ and
$\xymatrix@C=12pt{B & \ar[l]_-g C \ar@{=}[r] & C}$ belong to the class $\mathcal
S$ in Example \ref{ex:CoMonS} of spans in $\mathsf C$. Under this assumption the $\mathcal
S$-relative pullback \eqref{eq:coten-pullback} becomes isomorphic to the
so-called {\em cotensor product}; defined as the equalizer  
\begin{equation} \label{eq:cotensor}
\xymatrix@C=15pt{
A \coten B C \ar[r]^-j &
AC \ar@<2pt>[rr]^-{1f1.\delta 1}  \ar@<-2pt>[rr]_-{1g1.1\delta } &&
ABC}
\end{equation}
in $\mathsf M$, thanks to the serially commutative diagrams
$$
\xymatrix@C=45pt{
AC \ar@<2pt>[r]^-{\widehat{f\varepsilon}}\ar@<-2pt>[r]_-{\widehat{\varepsilon g}} \ar@{=}[d] &
ACBAC \ar[d]^-{1\varepsilon 1 \varepsilon 1} \\
AC \ar@<2pt>[r]^-{1f1.\delta 1}\ar@<-2pt>[r]_-{1g1.1\delta}  &
ABC}\qquad \qquad
\xymatrix@C=45pt{
AC \ar@<2pt>[r]^-{1f1.\delta 1}\ar@<-2pt>[r]_-{1g1.1\delta}  \ar@{=}[d] &
ABC \ar[d]^-{1c_{A,CB}1.11c1.\delta 1 \delta} \\
AC \ar@<2pt>[r]^-{\widehat{f\varepsilon}}\ar@<-2pt>[r]_-{\widehat{\varepsilon g}} &
ACBAC .}
$$

The current example can be considered in the particular situation when
$\mathsf M$ is a Cartesian symmetric monoidal category. Then
the class $\mathcal S$ of spans in Example \ref{ex:CoMonS} is the class of all
spans in $\mathsf C \cong \mathsf M$ and thus $\mathcal S$-relative pullbacks
are just usual pullbacks; see Example \ref{ex:pullback}. 
\end{example}

\begin{lemma} \label{lem:p_preserves_S}
For any admissible class $\mathcal S$ of spans in an arbitrary category
take an $\mathcal S$-relative pullback \eqref{eq:S-pullback}. The
following assertions hold.
\begin{itemize}
\item[{(1)}] If
$\xymatrix@C=12pt{A' & \ar[l]_-a A \ar@{=}[r] & A}\in \mathcal S$ then also
$\xymatrix@C=12pt{A' & \ar[l]_-a A & \ar[l]_-{p_A} A \coten B C \ar@{=}[r] & 
A \coten B C}\in \mathcal S$.
\item[{(2)}] If
$\xymatrix@C=12pt{C \ar@{=}[r] & C  \ar[r]^-c & C'}\in \mathcal S$ then also
$\xymatrix@C=12pt{A \coten B C \ar@{=}[r] & A \coten B C \ar[r]^-{p_C} &
C \ar[r]^-c & C'}\in \mathcal S$.
\end{itemize}
\end{lemma}

\begin{proof}
We only prove part (1), part (2) follows analogously. 
By assumption the span 
$\xymatrix@C=12pt{A' & \ar[l]_-a A \ar@{=}[r] & A}$ is in $\mathcal S$
hence by (PRE)
$$
\xymatrix@C=12pt{A' & \ar[l]_-a A & \ar[l]_-{p_A} A \coten B C \ar[r]^-{p_A}
  & A}\in \mathcal S.
$$ 
By construction
$\xymatrix@C=12pt{A & \ar[l]_-{p_A} A \coten B C \ar[r]^-{p_C}
  & C}\in \mathcal S$. Then by (POST)
$$
\xymatrix@C=12pt{A' & \ar[l]_-a A & \ar[l]_-{p_A} A \coten B C \ar[r]^-{p_C}
  & C}\in \mathcal S.
$$
By the reflection property of $A\coten B C$ the displayed conditions imply the
claim.
\end{proof}

\begin{proposition}\label{prop:S-pullback_morphisms}
Let $\mathcal S$ be an admissible class of spans in an arbitrary category. Consider $\mathcal S$-relative pullbacks
$$
\xymatrix{
A \coten B C \ar[r]^-{p_C} \ar[d]_-{p_A} &
C \ar[d]^-g \\
A \ar[r]_-f &
B}
\qquad 
\xymatrix{
A' \coten {B'} C' \ar[r]^-{p_{C'}} \ar[d]_-{p_{A'}} &
C' \ar[d]^-{g'} \\
A' \ar[r]_-{f'} &
B'}
\qquad 
\xymatrix{
A^{\prime\prime} \coten {B^{\prime\prime}} C^{\prime\prime} \ar[r]^-{p_{C^{\prime\prime}}} \ar[d]_-{p_{A^{\prime\prime}}} &
C^{\prime\prime} \ar[d]^-{g^{\prime\prime}} \\
A^{\prime\prime} \ar[r]_-{f^{\prime\prime}} &
B^{\prime\prime}.}
$$
\begin{itemize}
\item[{(1)}] 
For any morphisms
$\xymatrix@C=15pt{A \ar[r]^-a & A'}$, $\xymatrix@C=15pt{B \ar[r]^-b & B'}$ and
$\xymatrix@C=15pt{C \ar[r]^-c & C'}$ such that $b.f=f'.a$ and $b.g=g'.c$,
there is a unique morphism $a\morcoten c$ rendering commutative
$$
\xymatrix@C=15pt@R=15pt{
A\coten B C \ar[rr]^-{p_C} \ar[dd]_-{p_A} \ar@{-->}[rd]^-{a\diagcoten c} &&
C \ar[d]^-c \\
& A' \coten {B'} C' \ar[r]^-{p_{C'}} \ar[d]_-{p_{A'}} &
C'\ar[d]^-{g'} \\
A \ar[r]_-a &
A' \ar[r]_-{f'} &
B'.}
$$
\item[{(2)}] The operation $\Box$ of part (1) is functorial in the sense that for any further morphisms 
$\xymatrix@C=15pt{A' \ar[r]^-{a'} & A^{\prime\prime}}$, $\xymatrix@C=15pt{B'
  \ar[r]^-{b'} & B^{\prime\prime}}$ and $\xymatrix@C=15pt{C' \ar[r]^-{c'} &
  C^{\prime\prime}}$ such that $b'.f'=f^{\prime\prime}.a'$ and
  $b'.g'=g^{\prime\prime}.c'$, the equality  
$(a' \morcoten c').(a\morcoten c)=a'.a \morcoten c'.c$ holds.
\end{itemize}
\end{proposition}

\begin{proof}
(1) By construction 
$\xymatrix@C=12pt{A & \ar[l]_-{p_A} A \coten B C \ar[r]^-{p_C} & C}$ belongs
to $\mathcal S$. Hence by property (POST), also
$\xymatrix@C=12pt{A' & \ar[l]_-a A & \ar[l]_-{p_A} A \coten B C \ar[r]^-{p_C}
  & C \ar[r]^-c & C'}$ belongs to $\mathcal S$.
Therefore the stated morphism $a\morcoten c$ exists by the universality of the
$\mathcal S$-relative pullback $A' \coten {B'} C'$ and the commutativity of
$$
\xymatrix@R=15pt@C=30pt{
A \coten B C \ar[d]_-{p_A} \ar[r]^-{p_C} &
C \ar[d]^-g \ar[r]^-c &
C' \ar[dd]^-{g'} \\
A \ar[d]_-a \ar[r]^-f &
B \ar[rd]^-b \\
A' \ar[rr]_-{f'} &&
B'.}
$$

(2) Both morphisms $(a' \morcoten c').(a\morcoten c)$ and $a'.a \morcoten c'.c$
render commutative the same diagram
$$
\xymatrix@C=15pt@R=12pt{
A \coten B C \ar[rrr]^-{p_C} \ar[ddd]_-{p_A} \ar@{-->}[rrdd] &&&
C \ar[d]^-c \\
&&& C' \ar[d]^-{c'} \\
&& A^{\prime\prime} \coten {B^{\prime\prime}} C^{\prime\prime}
\ar[r]^-{p_{C^{\prime\prime}}} \ar[d]_-{p_{A^{\prime\prime}}} &
C^{\prime\prime} \ar[d]^-{g^{\prime\prime}} \\
A \ar[r]_-a &
A'\ar[r]_-{a'} &
A^{\prime\prime} \ar[r]_-{f^{\prime\prime}}&
B^{\prime\prime}.}
$$
Hence they are equal by the universality of $A^{\prime\prime} \coten {B^{\prime\prime}} C^{\prime\prime}$.
\end{proof}

\begin{proposition} \label{prop:S-pullback_unital&associative}
For any admissible class $\mathcal S$ of spans in an arbitrary category, 
%the $\mathcal S$-relative pullback operation is associative and unital. More
%explicitly, 
the following assertions hold.
\begin{itemize}
\item[{(1)}] If 
$\xymatrix@C=12pt{A\ar@{=}[r] & A \ar[r]^-{f} & B}\in \mathcal S$ then the
first diagram below is an $\mathcal S$-relative pullback and if 
$\xymatrix@C=12pt{B & \ar[l]_-g C\ar@{=}[r] & C}\in \mathcal S$
then the second diagram is so.
$$
\xymatrix{
A \ar[r]^-f \ar@{=}[d] &
B \ar@{=}[d] \\
A \ar[r]_-f & B}
\qquad \qquad
\xymatrix{
C \ar[d]_-g \ar@{=}[r] &
C \ar[d]^-g \\
B \ar@{=}[r] & 
B}
$$
That is to say, $\xymatrix@C=12pt{A \coten B B\ar[r]^-{p_A} & A}$ and 
$\xymatrix@C=12pt{B\coten B C\ar[r]^-{p_C} & C}$ are isomorphisms. 
\item[{(2)}] Consider morphisms 
$\xymatrix@C=12pt{A\ar[r]^- f & B & \ar[l]_-g C \ar[r]^-h & D & \ar[l]_-k E}$
such that all of the $\mathcal S$-relative pullbacks 
$$
\xymatrix{
A \coten B C \ar[r]^-{p_C} \ar[d]_-{p_A} &
C \ar[d]^-g\\
A \ar[r]_-f &
B} \quad
\xymatrix@R=20pt{
(A \coten B C)\coten D E  \ar[r]^-{p_E} \ar[d]_-{p_{A\diagcoten_B  C}} &
E \ar[d]^-k\\
A \coten B C \ar[r]_-{h.p_C} &
D} \quad
\xymatrix{
C \coten D E \ar[r]^-{p_E} \ar[d]_-{p_C} &
E \ar[d]^-k\\
C \ar[r]_-h &
D} \quad
\xymatrix{
A \coten B (C\coten D E)  \ar[d]_-{p_A} \ar[r]^-{p_{C\diagcoten_D  E}} &
C \coten D E \ar[d]^-{g.p_C} \\
A \ar[r]_-f &
B}
$$
exist. Then $(A\coten B C)\coten D E$ and $A\coten B (C \coten D E)$ are isomorphic.
\item[{(3)}] Consider morphisms 
$\xymatrix@C=12pt{A\ar[r]^- f & B & \ar[l]_-g C}$ such that the first listed
  $\mathcal S$-relative pullback $A\coten B C$ in part (2) exists. Then the
  isomorphisms of part (1) and the isomorphism $(A\coten B B) \coten B C \to
  A\coten B (B \coten B C)$ of part (2) satisfy Mac Lane's triangle condition.
\item[{(4)}] Consider morphisms 
$\xymatrix@C=12pt{A\ar[r] & B & \ar[l]C \ar[r] & D & \ar[l] E \ar[r] & F & \ar[l] G}$
such that all of the  $\mathcal S$-relative pullbacks $A\coten B C \coten D E \coten F G$ with any (hence by part (2) all) possible bracketing exist. Then the isomorphisms of part (2) satisfy Mac Lane's pentagon condition. 
\end{itemize}
\end{proposition}
To the question of the existence of the $\mathcal S$-relative pullbacks in
parts (2) and (4) of Proposition \ref{prop:S-pullback_unital&associative} we
shall return in Proposition \ref{prop:iterated_S-pullback}.

\begin{proof}
Part (1) is obvious. For part (2) note that by part (1) of Proposition
\ref{prop:S-pullback_morphisms} the top row of the commutative diagram
$$
\xymatrix@R=15pt@C=40pt{
(A\coten B C)\coten D E \ar[r]^-{p_C \diagcoten 1} \ar[d]_-{p_{A \diagcoten_B C}} &
C \coten D E \ar[d]^-{p_C} \\
A \coten B C \ar[r]^-{p_C} \ar[d]_-{p_A} &
C \ar[d]^-g \\
A \ar[r]_-f &
B}
$$
is well-defined.
By construction 
$$
\xymatrix@C=18pt{
A & \ar[l]_-{p_A} A\coten B C \ar[r]^-{p_C} & C}
\quad \textrm{and}\quad
\xymatrix@C=15pt{
A\coten B C && \ar[ll]_-{p_{A \diagcoten_B C}} (A\coten B C)\coten D E
\ar[r]^-{p_E} & E}
$$ 
belong to $\mathcal S$. Then by properties (PRE) and (POST) of $\mathcal S$,
respectively, also the spans
$$
\xymatrix@R=1pt{
A & 
\ar[l]_-{p_A} A\coten B C  & 
\ar[l]_-{p_{A \diagcoten_B C}} (A\coten B C)\coten D E
\ar[r]^-{p_{A\diagcoten_B C}}  \ar@/_.8pc/[rd]_-{p_C \diagcoten 1} &
A \coten B C \ar[r]^-{p_C} &
C  \\
&&& C \coten D E \ar@/_.8pc/[ru]_-{p_C} \\
\\
A & 
\ar[l]_-{p_A} A\coten B C  & 
\ar[l]_-{p_{A \diagcoten_B C}} (A\coten B C)\coten D E \ar[rr]^-{p_E}
\ar@/_.8pc/[rd]_-{p_C \diagcoten 1} &&
E\\
&&& C \coten D E \ar@/_.8pc/[ru]_-{p_E}}
$$
belong to $\mathcal S$. Hence we conclude by the reflection property of $C
\coten D E$ that 
$$
\xymatrix{
A & 
\ar[l]_-{p_A} A\coten B C  & 
\ar[l]_-{p_{A \diagcoten_B C}} (A\coten B C)\coten D E
\ar[r]^-{p_C \diagcoten 1} &
C \coten D E}
$$
belongs to $\mathcal S$. With all that information at hand, there is a unique
morphism $l$ rendering commutative the first diagram of
$$
\xymatrix{
(A\coten B C)\coten D E
\ar[dd]_-{p_{A\diagcoten_B C}}
\ar@/^2pc/[rrd]^-{p_C \diagcoten 1}
\ar@{-->}[rd]^-l \\
& A\coten B (C \coten D E) \ar[r]^-{p_{C\diagcoten_D E}} \ar[d]_-{p_A} &
C \coten D E \ar[d]^-{g.p_C} \\
A\coten B C \ar[r]_-{p_A} &
A \ar[r]_-f &
B}\quad
\xymatrix{
A\coten B (C\coten D E)
\ar[rr]^-{p_{C\diagcoten_D E}}
\ar@/_2pc/[rdd]_-{1 \diagcoten p_C}
\ar@{-->}[rd]^-{\widetilde l} &&
C\coten D E \ar[d]^-{p_E} \\
& (A\coten B C) \coten D E \ar[d]_-{p_{A\diagcoten_B C}} \ar[r]^-{p_E} &
E \ar[d]^-k \\
& A \coten B C \ar[r]_-{h.p_C} &
D}
$$
A symmetric reasoning yields a morphism $\widetilde l$ in the second
diagram. Since their right verticals are joint monomorphisms,
commutativity of both diagrams 
$$
\xymatrix@C=18pt{
& (A\coten B C)\coten D E \ar[r]^-{p_C \diagcoten 1} 
\ar[d]^-{p_{A\diagcoten_B C}} &
C \coten D E \ar[d]^-{p_C} \\
A\coten B (C \coten D E) \ar@/^1.5pc/[ru]^-{\widetilde l}
\ar[r]^-{1 \diagcoten p_C}
\ar@/_1.6pc/[rrd]_-{p_{C\diagcoten_D E}} &
A\coten B C \ar[r]^-{p_C} & 
C \\
&& C\coten D E \ar[u]_-{p_C}}
\quad
\xymatrix@C=18pt{
& (A\coten B C)\coten D E \ar[r]^-{p_C \diagcoten 1} \ar[rd]_-{p_E} &
C \coten D E \ar[d]^-{p_E} \\
A\coten B (C \coten D E) \ar@/^1.5pc/[ru]^-{\widetilde l}
\ar@/_1.6pc/[rrd]_-{p_{C\diagcoten_D E}} && 
E \\
&& C\coten D E \ar[u]_-{p_E}}
$$
proves $(p_C \morcoten 1).\widetilde l=p_{C\diagcoten_D E}$. This is used to
see the commutativity of the second diagram of
$$
\xymatrix@C=12pt{
& (A\coten B C)\coten D E \ar[r]^-l \ar[d]^-{p_{A\diagcoten_B C}} &
A\coten B (C\coten D E) \ar[d]^-{p_A} \\
A\coten B (C\coten D E)\ar[r]^-{1\diagcoten p_C} \ar@/^1.5pc/[ru]^-{\widetilde l}
\ar@{=}@/_1.6pc/[rrd] &
A \coten B C \ar[r]^-{p_A} &
A \\
&& 
A\coten B (C\coten D E) \ar[u]_-{p_A}}
\quad
\xymatrix@C=12pt{
& (A\coten B C)\coten D E \ar[r]^-l \ar[rd]_-{p_C \diagcoten 1} &
A\coten B (C\coten D E)  \ar[d]^-{p_{C\diagcoten_D E}}\\
A\coten B (C\coten D E) \ar@/^1.5pc/[ru]^-{\widetilde l} \ar@{=}@/_1.6pc/[rrd] &&
C \coten D E \\
&& 
A\coten B (C\coten D E). \ar[u]_-{p_{C\diagcoten_D E}}}
$$
Since their right verticals are joint monomorphisms, the commutativity of these
diagrams implies $l.\widetilde l=1$. A symmetric reasoning leads to
$\widetilde l.l=1$ so that $l$ and $\widetilde l$ are mutual inverses. 

(3) Since 
$\xymatrix@C=12pt{A & \ar[l]_-{p_A} A\coten B C \ar[r]^-{p_C} & C}$ are joint monomorphisms, the claim follows by the commutativity of both diagrams below.
$$
\xymatrix@R=25pt@C=15pt{
&& A\coten B C \ar[d]^-{p_A} \\
(A\coten B B)\coten B C \ar[r]^-{p_{A\diagcoten_B B}} \ar@/^1.2pc/[rru]^-{p_A \diagcoten 1} \ar@/_1.1pc/[rd]_-l &
A \coten B B \ar[r]^-{p_A } &
A \\
& A\coten B (B \coten B C) \ar[r]_-{1\diagcoten p_C} \ar[ru]^-{p_A} &
A \coten B C \ar[u]_-{p_A}} 
\quad
\xymatrix@R=9pt@C=10pt{
&&& A\coten B C \ar[d]^-{p_C} \\
(A\coten B B)\coten B C \ar[rrr]^-{p_C} \ar@/^1.2pc/[rrru]^-{p_A \diagcoten 1} \ar@/_1.1pc/[rdd]_-l \ar[rrd]_-{p_B \diagcoten 1} &&&
C  \\
&& B \coten B C \ar[ru]_-{p_C} \\
& A\coten B (B \coten B C) \ar[rr]_-{1\diagcoten p_C} \ar[ru]^-{p_{B \diagcoten_B C }} &&
A \coten B C \ar[uu]_-{p_C}} 
$$
(4) The claim follows by similar standard arguments; using the construction of
$l$ and the fact that 
$$
\xymatrix@R=10pt@C=30pt{
A \coten B (C\coten D (E \coten F G)) \ar[r]^-{p_A} & A\\
A \coten B (C\coten D (E \coten F G)) \ar[rr]^-{p_{C\diagcoten_D (E \diagcoten_F G)}} &&
C\coten D (E \coten F G) \ar[r]^-{p_C} & C \\
A \coten B (C\coten D (E \coten F G)) \ar[rr]^-{p_{C\diagcoten_D (E \diagcoten_F G)}} &&
C\coten D (E \coten F G) \ar[r]^-{p_{E \diagcoten_F G}} &
E \coten F G \ar[r]^-{p_E} & E \\
A \coten B (C\coten D (E \coten F G)) \ar[rr]^-{p_{C\diagcoten_D (E \diagcoten_F G)}} &&
C\coten D (E \coten F G) \ar[r]^-{p_{E \diagcoten_F G}} &
E \coten F G \ar[r]^-{p_G} & G
}
$$
are jointly monic.
\end{proof}

Since $\mathcal S$-relative pullbacks are defined up-to isomorphisms,
Proposition \ref{prop:S-pullback_unital&associative} allows us to pretend that
$\Box$ is associative and omit the parentheses as well as the isomorphisms $l$
in Proposition \ref{prop:S-pullback_unital&associative}.

\begin{proposition} \label{prop:S-pullback_monoid}
For a monoidal admissible class $\mathcal S'$ of spans in a monoidal category
$\mathsf M$, consider an $\mathcal S'$-relative pullback
\begin{equation} \label{eq:S-pullback_monoid}
\xymatrix{
A \coten B C \ar[r]^-{p_C} \ar[d]_-{p_A} &
C \ar[d]^-g \\
A \ar[r]_-f &
B}
\end{equation}
in which $f$ and $g$ are monoid morphisms.
\begin{itemize}
\item[{(1)}] There is a unique monoid structure on $A \coten B C$ such that
  $p_A$ and $p_C$ are monoid morphisms.
\item[{(2)}] The diagram of \eqref{eq:S-pullback_monoid} is an $\mathcal
  S$-relative pullback with respect to the admissible class $\mathcal S$ of
  spans in the category of monoids in $\mathsf M$ defined in Example
  \ref{ex:Mon(M)S}. 
\end{itemize}
\end{proposition}

\begin{proof}
(1) By construction 
$\xymatrix@C=12pt{A & \ar[l]_-{p_A} A\coten B C \ar[r]^-{p_C} &C}\in \mathcal S$
hence by the multiplicativity of $\mathcal S$
$\xymatrix@C=16pt{A^2 & \ar[l]_-{p_Ap_A} (A\coten B C)^2 \ar[r]^-{p_Cp_C} &C^2}\in
\mathcal S$. Then by (POST)
$\xymatrix@C=16pt{
A & \ar[l]_-m A^2 & \ar[l]_-{p_Ap_A} (A\coten B C)^2 \ar[r]^-{p_Cp_C} &C^2
\ar[r]^-m & A}\in \mathcal S$. Hence by the commutativity of the first diagram
in
\begin{equation} \label{eq:S-pullback_m&u}
\xymatrix@C=24pt@R=24pt{
(A\coten B C)^2 \ar[d]_-{p_Ap_A} \ar[r]^-{p_Cp_C} & 
C^2 \ar[r]^-m \ar[d]^-{gg} &
C \ar[dd]^-g \\
A^2 \ar[r]_-{ff} \ar[d]_-m &
B^2 \ar[rd]_-m \\
C \ar[rr]_-f &&
B}\quad
\xymatrix@C=20pt@R=20pt{
(A\coten B C)^2 \ar[rr]^-{p_Cp_C} \ar[dd]_-{p_Ap_A} \ar@{-->}[rd]^-m &&
C^2 \ar[d]^-m \\
& A\coten B C\ar[r]^-{p_C}\ar[d]_-{p_A} &
C \ar[d]^-g \\
A^2 \ar[r]_-m &
A \ar[r]_-f &
B}
\quad
\xymatrix@C=24pt@R=24pt{
I \ar@/^1.5pc/[rrd]^-u \ar@/_1.5pc/[rdd]_-u \ar@{-->}[rd]^-u & \\
& A\coten B C\ar[r]^-{p_C}\ar[d]_-{p_A} &
C \ar[d]^-g \\
& A \ar[r]_-f &
B}
\end{equation}
there is a unique filler $m$ for the second diagram of
\eqref{eq:S-pullback_m&u}.
By the unitality of $\mathcal S$, the span 
$\xymatrix@C=12pt{A & \ar[l]_-u I \ar[r]^-u & C}$
belongs to $\mathcal S$. Then by $f.u=u=g.u$, there is a unique filler $u$ for
the third diagram of \eqref{eq:S-pullback_m&u}.
By a standard reasoning, associativity and unitality of the monoid $A \coten B
C$ follows from the respective properties of $A$ and $C$ making use of the
fact that 
$\xymatrix@C=12pt{A & \ar[l]_-{p_A} A\coten B C \ar[r]^-{p_C} &C}$
are joint monomorphisms.

(2) The span of monoids 
$\xymatrix@C=12pt{A & \ar[l]_-{p_A} A\coten B C \ar[r]^-{p_C} &C}$
belongs to $\mathcal S$ by definition and the square of
\eqref{eq:S-pullback_monoid} commutes by construction. 
In order to see its universality, take a span of monoids
$\xymatrix@C=12pt{A & \ar[l]_-{a} D  \ar[r]^-c &C}$ in $\mathcal S$ such that
$f.a=g.c$. Then by definition 
$\xymatrix@C=12pt{A & \ar[l]_-{a} D  \ar[r]^-c &C}\in \mathcal S'$.
Since \eqref{eq:S-pullback_monoid} is an $\mathcal S'$-relative pullback in
$\mathsf M$, there is a unique filler $d$ of the diagram
$$
\xymatrix@R=12pt@C=12pt{
D \ar@/^1.2pc/[rrrd]^-c\ar@/_1.2pc/[rddd]_-a \ar@{-->}[rd]^-d \\
& A \coten B C \ar[rr]^-{p_C}\ar[dd]_-{p_A} &&
C \ar[dd]^-g \\
\\
& A \ar[rr]_-f &&
B}
$$
in $\mathsf M$. Using again that 
$\xymatrix@C=12pt{A & \ar[l]_-{p_A} A\coten B C \ar[r]^-{p_C} &C}$
are joint monomorphisms in $\mathcal M$, the morphism $d$ is multiplicative by
the commutativity of
$$
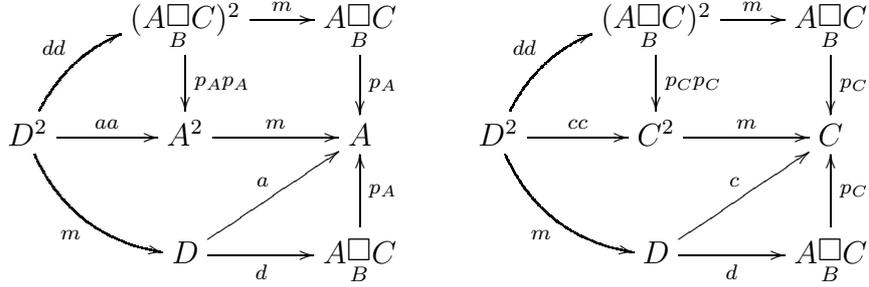
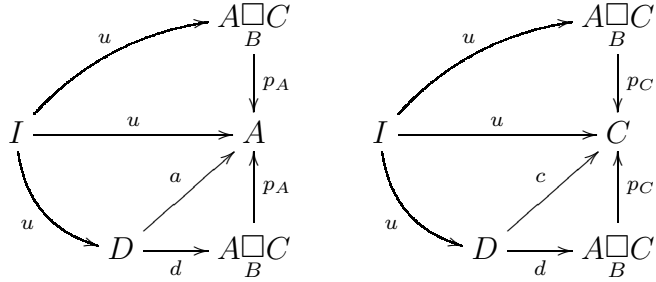
\xymatrix{
& (A \coten B C)^2 \ar[r]^-m \ar[d]^-{p_Ap_A} &
A \coten B C \ar[d]^-{p_A} \\
D^2 \ar@/^1pc/[ru]^-{dd} \ar[r]^-{aa} \ar@/_1pc/[rd]_-m &
A^2 \ar[r]^-m &
A\\
& D \ar[ru]^-a \ar[r]_-d &
A \coten B C \ar[u]_-{p_A}}
\qquad 
\xymatrix{
& (A \coten B C)^2 \ar[r]^-m \ar[d]^-{p_Cp_C} &
A \coten B C \ar[d]^-{p_C} \\
D^2 \ar@/^1pc/[ru]^-{dd} \ar[r]^-{cc} \ar@/_1pc/[rd]_-m &
C^2 \ar[r]^-m &
C\\
& D \ar[ru]^-c \ar[r]_-d &
A \coten B C \ar[u]_-{p_C}}
$$
and unital by
$$
\xymatrix{
&& A \coten B C \ar[d]^-{p_A} \\
I \ar@/^1pc/[rru]^-u \ar[rr]^-u \ar@/_1pc/[rd]_-u &&
A\\
& D \ar[ru]^-a \ar[r]_-d &
A \coten B C \ar[u]_-{p_A}}
\qquad 
\xymatrix{
&& A \coten B C \ar[d]^-{p_C} \\
I \ar@/^1pc/[rru]^-u \ar[rr]^-u \ar@/_1pc/[rd]_-u &&
C\\
& D \ar[ru]^-c \ar[r]_-d &
A \coten B C \ar[u]_-{p_C}}
$$
The reflection property is obviously inherited from $\mathsf M$.
\end{proof}

%%%%%%%%%%%%%%%%%%%%%%%%%%%%%%% SEC 4   %%%%%%%%%%%%%%%%%%%%%%%%%%%%%%

\section{Relative categories}
\label{sec:rel_cat}

\begin{assumption} \label{ass:S-pullback}
For an admissible class $\mathcal S$ of spans in some category $\mathsf C$ we
make the following assumption:
whenever a cospan
$\xymatrix@C=12pt{A \ar[r]^-f & B & \ar[l]_-g C}$
has legs in $\mathcal S$ (see Definition \ref{def:legs_in_S}), 
there exists their $\mathcal S$-relative pullback 
$\xymatrix@C=12pt{
A & \ar[l]_-{p_A}  A \coten B C \ar[r]^-{p_C} & C}$.
\end{assumption}

\begin{example}
If $\mathcal S$ is the class of all spans in some category $\mathsf C$, then
Assumption \ref{ass:S-pullback} reduces to the assumption that pullbacks exist
in $\mathsf C$.
\end{example}

\begin{example} \label{ex:ass_CoMon}
As in Example \ref{ex:CoMon-pullback}, let $\mathsf C$ be the category of
comonoids in a braided monoidal category $\mathsf M$ in which equalizers exist and are preserved by the monoidal product with any object. 
Then it is proven in Example \ref{ex:CoMon-pullback} that in $\mathsf C$ all pullbacks exist relative to the admissible class $\mathcal S$ in Example \ref{ex:CoMonS} of spans in $\mathsf C$. Thus in particular Assumption \ref{ass:S-pullback} holds for this class $\mathcal S$.
\end{example}

\begin{example} \label{ex:ass_Mon(M)}
Suppose that Assumption \ref{ass:S-pullback} holds for a monoidal admissible
class $\mathcal S'$ of spans in a monoidal category $\mathsf M$. Then it also
holds for the admissible class $\mathcal S$ of spans in the category of monoids in
$\mathsf M$ in Example \ref{ex:Mon(M)S}.

Indeed, if for some monoid morphisms 
$\xymatrix@C=12pt{
A \ar[r]^-f & B & \ar[l]_-g  C}$
the spans 
$\xymatrix@C=12pt{
A \ar@{=}[r] & A \ar[r]^-f & B}$
and
$\xymatrix@C=12pt{
B & \ar[l]_-g C \ar@{=}[r] & C}$ belong to $\mathcal S$, then by definition
they belong to $\mathcal S'$ too. Then by Assumption \ref{ass:S-pullback}
there exists their $\mathcal S'$ relative pullback in $\mathsf M$. And it is
an $\mathcal S$-relative pullback of monoids in $\mathsf M$ by Proposition
\ref{prop:S-pullback_monoid}. 
\end{example}

\begin{proposition}\label{prop:iterated_S-pullback}
If Assumption \ref{ass:S-pullback} holds for some admissible class $\mathcal S$ of spans in
an arbitrary category $\mathsf C$ then all of the $\mathcal S$-relative
pullbacks listed in part (2) of Proposition
\ref{prop:S-pullback_unital&associative} exist provided that the following spans 
belong to $\mathcal S$. 
$$
\xymatrix@C=12pt{
A \ar@{=}[r] & A \ar[r]^-f & B}
\quad
\xymatrix@C=12pt{
B & \ar[l]_-g C \ar@{=}[r] & C}
\quad
\xymatrix@C=12pt{
C \ar@{=}[r] & C \ar[r]^-h & D}
\quad
\xymatrix@C=12pt{
D & \ar[l]_-k E \ar@{=}[r] & E}
$$
\end{proposition}

\begin{proof}
Existence of the $\mathcal S$-relative pullbacks listed first and third in part
(2) of Proposition \ref{prop:S-pullback_unital&associative} immediately
follows by Assumption \ref{ass:S-pullback}. 
In order to see existence of the $\mathcal S$-relative pullback listed second,
use first that by the assumption that
$\xymatrix@C=12pt{C \ar@{=}[r] & C \ar[r]^-h & D}\in \mathcal S$
and by Lemma \ref{lem:p_preserves_S}~(2) also
$\xymatrix@C=18pt{A \coten B C \ar@{=}[r] & A \coten B C \ar[r]^-{h.p_C} & D}$
is in $\mathcal S$. Since 
$\xymatrix@C=12pt{D & \ar[l]_-k E \ar@{=}[r] & E}$ is in $\mathcal S$ by
assumption, the existence of the stated $\mathcal S$-relative pullback
$(A\coten B C)\coten D E$ follows by Assumption \ref{ass:S-pullback}. 
An analogous reasoning applies to the $\mathcal S$-relative pullback $A\coten
B (C \coten D E)$ listed last in part (2) of Proposition
\ref{prop:S-pullback_unital&associative}. 
\end{proof}

\begin{corollary} \label{cor:S-pullback_moncat}
Consider an admissible class $\mathcal S$ of spans in an arbitrary category $\mathsf C$ for which Assumption \ref{ass:S-pullback} holds. For any object $B$ in $\mathsf C$ for which
$\xymatrix@C=10pt{B \ar@{=}[r] & B \ar@{=}[r] & B} \in \mathcal S$,
there is a monoidal category whose\\
\underline{objects} are spans 
$\xymatrix@C=12pt{B & \ar[l]_-t A \ar[r]^-s & B}$ which have their 
legs in $\mathcal S$ (cf. Definition \ref{def:legs_in_S})
\\
\underline{morphisms} are the morphisms of spans over $B$ \\
\underline{monoidal product} of 
$\xymatrix@C=12pt{B & \ar[l]_-t A \ar[r]^-s & B}$ and
$\xymatrix@C=12pt{B & \ar[l]_-{t'} A' \ar[r]^-{s'} & B}$ is
$\xymatrix@C=18pt{B & \ar[l]_-{t.p_A} A \coten B A' \ar[r]^-{s'.p_{A'}} & B}$
(its legs are in $\mathcal S$ by Lemma \ref{lem:p_preserves_S}) where $A\coten
B A'$ is the $\mathcal S$-relative pullback of  
$\xymatrix@C=12pt{A  \ar[r]^-s & B & \ar[l]_-{t'}  A'}$  \\
\underline{monoidal unit} is $\xymatrix@C=10pt{B \ar@{=}[r] & B \ar@{=}[r] &
  B}$.
\end{corollary}

For any positive integer $n$, we denote by $A^{\expcoten B n}$ the $n$'th monoidal
power of the object $\xymatrix@C=12pt{B & \ar[l]_-t A \ar[r]^-s & B}$. We
adopt the convention $A^{\expcoten B 0}=B$.

\begin{proof}
For morphisms of spans, the $\mathcal S$-relative pullback in Proposition
\ref{prop:S-pullback_morphisms} is obviously a morphisms of spans.  
So in view of Proposition \ref{prop:S-pullback_morphisms}, Proposition
\ref{prop:S-pullback_unital&associative} and Proposition
\ref{prop:iterated_S-pullback}, we only need to check the naturality of the
unit and associativity constraints in Proposition
\ref{prop:S-pullback_unital&associative}.
Naturality of the unit constraints --- that is, commutativity of
$$
\xymatrix{
A \coten B B \ar[r]^-{p_A} \ar[d]_-{a\diagcoten 1} &
A \ar[d]^-a \\
A' \coten B B \ar[r]_-{p_{A'}} &
A'} \qquad
\xymatrix{
B \coten B C \ar[r]^-{p_C} \ar[d]_-{1\diagcoten c} &
C \ar[d]^-c \\
B \coten B C' \ar[r]_-{p_{C'}} &
C'}
$$
for any morphisms of spans $\xymatrix@C=12pt{A \ar[r]^-a & A'}$ and 
$\xymatrix@C=12pt{C \ar[r]^-c & C'}$ --- holds by construction.  

For any morphisms of spans 
$\xymatrix@C=12pt{A \ar[r]^-a & A'}$,  
$\xymatrix@C=12pt{C \ar[r]^-c & C'}$ and 
$\xymatrix@C=12pt{E \ar[r]^-e & E'}$,
let us compose both $(a\morcoten (c\morcoten e)).l$ and $l.((a\morcoten
c)\morcoten e)$ with the joint monomorphisms 
$$
\xymatrix@C=15pt{A' \coten {B'} (C' \coten {D'} E') \ar[r]^-{p_{A'}} & A'}
\xymatrix@C=15pt{A' \coten {B'} (C' \coten {D'} E') 
\ar[r]^-{p_{C' \diagcoten_{D'} E'}} & 
C' \coten {D'} E' \ar[r]^-{p_{C'}} & C'}
\xymatrix@C=15pt{A' \coten {B'} (C' \coten {D'} E') 
\ar[r]^-{p_{C' \diagcoten_{D'} E'}} & 
C' \coten {D'} E' \ar[r]^-{p_{E'}} & E'}.
$$
The resulting pairs of composite morphisms are easily seen to be equal to 
$$
\xymatrix@C=9pt{(A \coten B C) \coten D E 
\ar[r]^-{\raisebox{8pt}{${}_{p_{A \diagcoten_B C}}$}} &
A \coten B C \ar[r]^-{\raisebox{8pt}{${}_{p_A}$}} & 
A \ar[r]^-{\raisebox{8pt}{${}_a$}} & A'}
\xymatrix@C=9pt{(A \coten B C) \coten D E 
\ar[r]^-{\raisebox{8pt}{${}_{p_{A \diagcoten_B C}}$}} &
A \coten B C \ar[r]^-{\raisebox{8pt}{${}_{p_C}$}} & 
C \ar[r]^-{\raisebox{8pt}{${}_c$}} & C'}
\xymatrix@C=9pt{(A \coten B C) \coten D E 
\ar[r]^-{\raisebox{8pt}{${}_{p_E}$}} &
E \ar[r]^-{\raisebox{8pt}{${}_e$}} & E',}
$$
respectively. 
This proves the naturality of the associativity constraint.
\end{proof}

It may happen that in some category not only those cospans have 
pullbacks relative to some class $\mathcal S$ of spans whose legs are in
$\mathcal S$. (Recall from Example \ref{ex:CoMon-pullback} that in certain
categories of comonoids all pullbacks exist relative to the class of spans 
in Example \ref{ex:CoMonS}). However, the monoidal structure of Corollary
\ref{cor:S-pullback_moncat} is available only on the category of those spans
whose legs are in $\mathcal S$; see Proposition
\ref{prop:S-pullback_unital&associative}~(1).  

\begin{example} \label{ex:pullback-moncat}
If $\mathcal S$ is the class of all spans in a category $\mathsf C$ having
pullbacks, then Corollary \ref{cor:S-pullback_moncat} describes the monoidal
category of spans in $\mathsf C$ via the usual pullback.    
\end{example}

\begin{example} \label{ex:CoMon-moncat}
As in Example \ref{ex:CoMon-pullback}, let $\mathsf C$ be the category of
comonoids in a braided monoidal category $\mathsf M$ in which equalizers exist and are preserved by the monoidal product with any object. Then we know from Example
\ref{ex:ass_CoMon} that Assumption \ref{ass:S-pullback} holds for the
admissible class $\mathcal S$ in Example \ref{ex:CoMonS} of spans in $\mathsf C$.
For a comonoid $B$ in $\mathsf M$ the condition 
$\xymatrix@C=10pt{B \ar@{=}[r] & B \ar@{=}[r] & B} \in \mathcal S$
reduces to the cocommutativity of the comonoid $B$. 
So by Corollary \ref{cor:S-pullback_moncat} the category of spans of comonoids
over a cocommutative comonoid $B$ with legs in $\mathcal S$ is monoidal via
the $B$-cotensor product of \eqref{eq:cotensor}. 
\end{example}

\begin{definition} \label{def:S-cat}
Consider an admissible class $\mathcal S$ of spans in an arbitrary category $\mathsf C$ for which Assumption \ref{ass:S-pullback} holds, and  an object $B$ in $\mathsf C$ for which 
$\xymatrix@C=10pt{B \ar@{=}[r] & B \ar@{=}[r] & B} \in \mathcal S$.
An {\em $\mathcal S$-relative category} with object of objects $B$ is a monoid in the monoidal category of Corollary \ref{cor:S-pullback_moncat}.
Explicitly, this means the data in
\begin{equation} \label{eq:rel-cat}
\xymatrix@C=40pt{
B \ar@{ >->}[r]|(.55){\, i\, } &
A  \ar@<-4pt>@{->>}[l]_-s \ar@<4pt>@{->>}[l]^t &
A \coten B A \ar[l]_-d}
\end{equation}
subject to the following axioms.
\begin{itemize}
\item[{(a)}] The legs of 
$\xymatrix@C=12pt{B & \ar[l]_-t A \ar[r]^-s & B}$
are in $\mathcal S$ (so that the pullback 
$\xymatrix@C=12pt{A & \ar[l]_-{p_1} A\coten B A \ar[r]^-{p_2} & A}$ 
of 
$\xymatrix@C=12pt{A \ar[r]^-s & B & \ar[l]_-t A}$ relative to the class
$\mathcal S$ exists).
\item[{(b)}]  $i$ is a common section of $s$ and $t$ (that is,
$\xymatrix@C=20pt{
B \ar@{ >->}[r]|(.55){\, i\, } &
A  \ar@<-4pt>@{->>}[l]_-s \ar@<4pt>@{->>}[l]^t}$ is a {\em reflexive graph}).
\item[{(c)}] $t.d=t.p_1$ and $s.d=s.p_2$. 
\item[{(d)}] $d.(i\morcoten 1)=1=d.(1\morcoten i)$.
\item[{(e)}] $d.(d\morcoten 1)=d.(1\morcoten d)$.
\end{itemize}
\end{definition}

\begin{definition} \label{def:S-functor}
Consider an admissible class $\mathcal S$ of spans in an arbitrary category $\mathsf C$ for which Assumption \ref{ass:S-pullback} holds, and  an object $B$ in $\mathsf C$ for which 
$\xymatrix@C=10pt{B \ar@{=}[r] & B \ar@{=}[r] & B} \in \mathcal S$.
An {\em $\mathcal S$-relative functor} between $\mathcal S$-relative
categories as in \eqref{eq:rel-cat} consists of a pair of morphisms 
$(\xymatrix@C=12pt{B \ar[r]^-b & B'},\xymatrix@C=12pt{A \ar[r]^-a & A'})$
which is 
\begin{itemize}
\item[{(a)}] a morphism of spans in the sense that $b.s=s'.a$ and $b.t=t'.a$
  (hence by Proposition \ref{prop:S-pullback_morphisms} there exists the
  $\mathcal S$-relative pullback morphism
$\xymatrix@C=16pt{A\coten B A \ar[r]^-{a\diagcoten a} &A'\coten {B'} A'}$)
\item[{(b)}] it is compatible with the monoid structure in the sense that
  $a.i=i'.b$ and $a.d=d'.(a\morcoten a)$.
\end{itemize}
\end{definition}

\section*{Summary and outlook }

In this paper pullbacks were introduced relative to a chosen class of spans. On this class we made assumptions which allow for the pullback to define a monoidal structure on the category of spans with their `legs in this class'. Relative (to the above class of spans) categories were defined as monoids in the so obtained monoidal category. Non-trivial examples are presented in categories of comonoids in braided monoidal categories. 

All this is meant to be a preparation for a further analysis to be carried out in \cite{Bohm:Xmod_II} and \cite{Bohm:Xmod_III}. In these sequel papers we will apply this theory to categories of monoids in symmetric monoidal categories; that is, we consider relative categories of monoids. They will be shown to be equivalent to relative crossed modules of monoids (see \cite{Bohm:Xmod_II}) and to relative simplicial monoids of Moore length 1 (in \cite{Bohm:Xmod_III}). 

Again, interesting examples will arise from categories of comonoids in braided monoidal categories; whose monoids are known as bimonoids. Taking the full subcategory of Hopf monoids in a category of bimonoids, some recent results in the literature --- 
\cite{Aguiar,Villanueva,Majid,FariaMartins,Emir} 
--- will be placed in a broader context.

%%%%%%%%%%%%%%%%%%%%%%%%%%%%%%% BIBLIOGRAPHY   %%%%%%%%%%%%%%%%%%%%%%%%%%%%%%

\bibliographystyle{plain}

\end{document}